\documentclass[11pt]{article}
\usepackage{amsmath, array, amstext, amsthm, amssymb, amsfonts, latexsym, amscd, mathtools}
\DeclarePairedDelimiter\floor{\lfloor}{\rfloor}
\usepackage{authblk}
\usepackage{tikz}
\usetikzlibrary{shadings,patterns,decorations.pathreplacing}
\newcolumntype{C}{>{$}c<{$}}
\newtheorem{theorem}{Theorem}[section]

\newtheorem{lemma}[theorem]{Lemma}
\newtheorem{proposition}[theorem]{Proposition}
\newtheorem{definition}[theorem]{Definition}
\newtheorem{remark}[theorem]{Remark}

\title{\bf On vertex peripherians and Wiener index of graphs with fixed number of cut vertices}
\author{Dinesh Pandey} 
\affil {Lazaridis School of Business and Economics, Wilfrid Laurier University, 75 University Avenue West, Waterloo, Ontario, N2L3C5, Canada. \\
Email: dpandey@wlu.ca}
\date{}

\begin{document}

\maketitle
\begin{abstract}
The distance of a vertex in a graph is the sum of distances from that vertex to all other vertices of the graph. The Wiener index of a graph is the sum of distances between all its unordered pairs of vertices. A graph has been obtained that contains a vertex achieving the maximum distance among all graphs on $n$ vertices with fixed number of cut vertices. Further the graphs  having maximum Wiener index among all graphs on $n$ vertices with at most $3$ cut vertices have been characterised.\\

\noindent {\bf Key words:} cut vertex, distance, Wiener index, extremal problems  \\

\noindent {\bf AMS subject classification.}  05C09, 05C12, 05C35
\end{abstract}

\section{Introduction}\label{Intro}
Throughout the paper the graphs are simple, finite, connected and undirected. Let $G$ be a graph with vertex set $V(G)$ and the edge set $E(G)$. The edge joining the vertices $u$ and $v$ of $G$ is denoted by $uv$. A {\it cut vertex} in $G$ is a vertex whose removal makes the graph disconnected. A vertex of degree one in $G$ is called a {\it pendant vertex}. An edge containing a pendant vertex of $G$ is called a {\it pendant edge}. For two isomorphic graphs $G_1$ and $G_2$, we use the notation $G_1\cong G_2$. The path and the cycle on $n$ vertices are denoted by $P_n$ and $C_n$, respectively. The complete bipartite graph $K_{1,n-1}$ is called a {\it star} on $n$ vertices. A {\it block} in $G$ is a maximal $2-connected$ subgraph of $G$. A {\it pendant block} of $G$ is a block containing exactly one cut vertex of $G$. Two blocks of $G$ are {\it adjacent} if they share a cut vertex.  For $u,v\in V(G)$, the distance $d_G(u,v)$ or $d(u,v)$ is the number of edges in a shortest path joining $u$ and $v$. The {\it distance} $D_G(v)$ of the vertex $v$ in $G$ is defined as $D_G(v)=\sum_{u\in V(G)}{d(v,u)}$. A vertex having maximum distance in $G$ is called a {\it peripherian vertex} of $G$ (see \cite{Zelinka}, page 93).  The {\it Wiener index} of $G$ is defined as the sum of distances between all its unordered pairs of vertices and denoted by $W(G)$. From this definition it follows that $W(G)=\frac{1}{2}\sum_{v\in V(G)}{D_G(v)}$. Other terminologies which are used in this article and not defined here can be found in \cite{West}.

The Wiener index is the most studied topological index in graph theory. It was introduced by the Chemist H. Wiener in \cite{Wiener} in $1947$. In mathematical literature, \cite{Entringer} seems to be the first paper studying Wiener index. The graphs having extremal (maximal or minimal) Wiener index among various classes of graph have been studied  extensively in last two decades. Two such recent studies can be seen in \cite{Czabarka} and \cite{Dankelmann}.  Over certain classes of graphs, the problem of maximizing the Wiener index seems comparatively difficult than the corresponding minimization problem. We denote the set of all connected graph on $n$ vertices with $k$ cut vertices by $\mathfrak{C}_{n,k}$. In \cite{Pandey}, the graphs having minimum Wiener index in $\mathfrak{C}_{n,k}$ have been characterised. In this paper we find a graph containing a vertex peripherian which attains maximum distance among all vertex peripherians in $\mathfrak{C}_{n,k}$. Further in the obtained graph we identify the vertex  peripherians. Using this, we characterise the graphs having maximum Wiener index in $\mathfrak{C}_{n,k}$, $0\leq k\leq 3$. The problem remains open for graphs having more than $3$ cut vertices.

The paper is organised in the following way. In section \ref{Preliminaries}, some results from the literature and some expressions for distance of a vertex and Wiener index of  some specific graphs have been presented. In Section \ref{Special blocks}, special kind of  pendant vertices and pendant blocks are introduced which are useful for the study. Section \ref{Distance} presents the vertex peripherians which achieve maximum distances among all vertex peripherians in $\mathfrak{C}_{n,k}$. In Section \ref{Wiener index}, the graphs having maximum Wiener index in $\mathfrak{C}_{n,k}$ for $0\leq k \leq 3$ have been characterised.

 \section {Preliminaries}\label{Preliminaries}

The following lemma shows the effect of deleting an edge on the distance of a vertex and on the Wiener index of  a graph, and they follow from the definitions.
\begin{lemma} \label{edge effect}
Let $G$ be a graph and $e\in E(G)$ such that $G-e$ is connected. Then 
\begin{enumerate}
\item[$(i)$] for any $v\in V(G)$, $D_{G-e}(v)\geq D_G(v)$.
\item[$(ii)$] $W(G-e)> W(G)$.
\end{enumerate}
\end{lemma}

If $|V(G)|=n$, then $G$ has at most $n-2$ cut vertices and $P_n$, $n>2$ is the only graph with $n-2$ cut vertices. So we consider $\mathfrak{C}_{n,k}$ where $0\leq k\leq n-3$. Let $G$ be a graph and $w$ be a cut vertex of $G$. Then there always exist two subgraphs $G_1$ and $G_2$ (both on at least $2$ vertices) such that $G\cong G_1\cup G_2$ and $V(G_1)\cap V(G_2)=\{w\}$. The next lemma is frequently used in counting the distance of a vertex and the Wiener index of a graph, in $\mathfrak{C}_{n,k}, k\geq 1$.

\begin{lemma} [\cite{Bsv}, Lemma 1.1]\label{count}
Let $G$ be a graph and $w$ be a cut vertex in $G$. Let $G_1$ and $G_2$ be two subgraphs of $G$ such that $G\cong G_1 \cup G_2$ and $V(G_1) \cap V(G_2)= \{w\}$. Then
\begin{enumerate}
\item[$(i)$] \label{count-1} for any $v\in V(G_1)$, $D_G(v)=D_{G_1}(v)+(|V(G_2)|-1)d(v,w)+D_{G_2}(w)$ and
\item[$(ii)$] \label{count-2}$W(G)=W(G_1)+W(G_2)+(|V(G_1)|-1)D_{G_2}(w)+(|V(G_2)|-1)D_{G_1}(w)$.
\end{enumerate}
\end{lemma}

A graph is called minimally $2$-connected if it is $2$-connected and deleting any edge gives a graph which is not $2$-connected.
\begin{lemma}[\cite{Dirac}, Theorem 2] \label{min 2-connected}
A minimally $2$-connected graph with more than $3$ vertices is triangle free.
\end{lemma}

Let $G$ has maximum Wiener index over $\mathfrak{C}_{n,k}$, $n\geq 4$. Then by Lemma  \ref{edge effect}, blocks of $G$ are minimally $2$-connected and hence by Lemma \ref{min 2-connected} the blocks of size more than $3$ are triangle free. Let $B$ be a block of size $3$ in $G$ containing the cut vertex $w$ of $G$ . Then $B$ is isomorphic to a triangle $wxy$.  If $B$ contains exactly one cut vertex $w$, then $G-xy \in \mathfrak{C}_{n,k}$ and by Lemma \ref{edge effect} $(ii)$,  $W(G-xy)> W(G)$, a contradiction. If $B$ contains two cut vertices say $w$ and $x$, then $G-wy \in \mathfrak{C}_{n,k}$ and $W(G-wy)> W(G)$, a contradiction. If all three vertices $w, x$ and $y$ are cut vertices in $G$, then for any $e\in \{wx, xy, yw \}$, $G-e \in \mathfrak{C}_{n,k}$ and $W(G-e)> W(G)$, a contradiction. 

Similarly it can be shown that if $G\in \mathfrak{C}_{n,k}$, $n\geq 4$ and $v_0\in V(G)$ such that $D_G(v_0)=\max\{D_G(v): v\in V(G)\}$, then there exists a triangle free graph $G'\in \mathfrak{C}_{n,k}$ obtained by removing some edges (if necessary) from $G$ such that $D_{G'}(v_0)\geq D_G(v_0)$. So, we conclude the following.\\
\begin{remark}\label{Remark_WI}
If $G$ has maximum Wiener index over $\mathfrak{C}_{n,k}$, $n\geq 4$, then $G$ is triangle free.
\end{remark}

\begin{remark}\label{Remark_Distance}
Among all $G'$ satisfying  $D_{G'}(v_0)=\max\{D_G(v): G\in \mathfrak{C}_{n,k}, n\geq 4, v\in V(G)\}$, there exists one which is triangle free.
\end{remark}

We now recall the distance of vertices  and the Wiener indices of some known graphs which will be used in later sections. First, consider the path $P_n: v_1v_2\ldots v_n$. Then for the vertex $v_i , 1\leq i\leq n$, 
\begin{equation}\label{Path_vertex}
D_{P_n}(v_i)=\frac{i(i-1)}{2}+\frac{(n-i)(n-i-1)}{2}\leq \frac{n(n-1)}{2}=D_{P_n}(v_1)=D_{P_n}(v_n)
\end{equation}
and 
$$W(P_n)={n+1\choose3}.$$
 By $L_{n,g}$, we denote the graph in $\mathfrak{C}_{n,n-g}$ obtained by identifying a pendant vertex of $P_{n-g+1}$ with a vertex of $C_g$ (see Figure \ref{graph-Lnk}). Note that $L_{n,n-g}$ has $g$ cut vertices. The graph $C_{m_1,m_2}^n$ is defined for $n\geq m_1+m_2-1$ and $m_1, m_2\geq 3$ in \cite{Pandey}. For $n\geq m_1+m_2$, $C_{m_1,m_2}^n$ is the graph obtained by identifying one pendant vertex of the path $P_{n+2-(m_1+m_2)}$ with a vertex of $C_{m_1}$ and the other pendant vertex of $P_{n+2-(m_1+m_2)}$ with a vertex of $C_{m_2}$. For $n=m_1+m_2-1$, $C_{m_1,m_2}^n$ is the graph obtained by identifying a vertex of $C_{m_1}$ with a vertex of $C_{m_2}$ (see Figure \ref{graph_Cmn}). Note that $C_{m_1, m_2}^n$ has $n+2-(m_1+m_2)$ cut vertices. 

\begin{figure}[h!]
\begin{center}
\begin{tikzpicture}[scale=.7]
\draw (0,0) circle [radius= 1cm];
\filldraw (1,0) circle [radius=.5mm];
\filldraw (2,0) circle [radius= .5mm];
\filldraw (3.3,0) circle [radius=.5mm];
\filldraw (4.3,0) circle [radius= .5mm];
\draw [dash pattern=on 1pt off 2pt] (2,0)--(3.3,0);
\draw (1,0)--(2,0);
\draw(3.3,0)--(4.3,0);
\draw(0,0) node {$C_g$};
\draw [decorate,decoration={brace,amplitude=5pt,mirror},xshift=2pt,yshift=0pt]
(0.85,-0.2) -- (4.3,-0.2) node [black,midway,yshift=-0.4cm]
{\footnotesize $P_{n-g+1}$};
\end{tikzpicture}
\caption{The graph $L_{n, g}$}\label{graph-Lnk}
\end{center}
\end{figure}
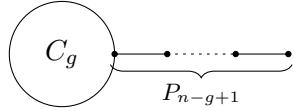

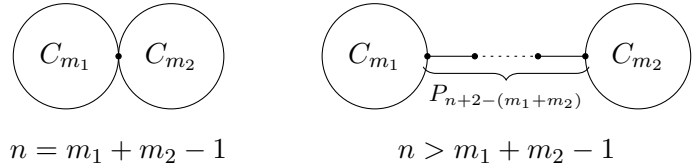
\begin{figure}[h!]
\begin{center}
\begin{tikzpicture}[scale =.7]
\draw (0,0)circle [radius= 1cm];
\draw (2,0) circle [radius= 1cm];
\draw  (1,-1.8) node {$n=m_1+m_2-1$};
\draw (0,0) node {$C_{m_1}$};
\draw (2,0) node {$C_{m_2}$};
\filldraw(1,0) circle [radius=.5mm];
\end{tikzpicture}
\hskip 1cm
\begin{tikzpicture}[scale=.7]
\draw (0,0)circle [radius= 1cm];
\draw (5,0) circle [radius= 1cm];
\filldraw (1,0)circle [radius= .5 mm];
\filldraw (1.9,0) circle [radius= .5 mm];
\filldraw (3.1,0)circle [radius= .5 mm];
\filldraw (4,0) circle [radius= .5 mm];
\draw [dash pattern= on 1pt off 2pt] (1.9,0)--(3.1,0);
\draw (1,0)--(1.9,0);
\draw(3.1,0)--(4,0);
\draw  (2.5,-1.8) node {$n>m_1+m_2-1$};
\draw (0,0) node {$C_{m_1}$};
\draw (5,0) node {$C_{m_2}$};
\draw [decorate,decoration={brace,amplitude=5pt,mirror},xshift=2pt,yshift=0pt]
(0.85,-0.2) -- (4,-0.2) node [black,midway,yshift=-0.4cm]
{\footnotesize $P_{n+2-(m_1+m_2)}$};
\end{tikzpicture}
\caption{The graphs $C_{m_1,m_2}^n$}\label{graph_Cmn}
\end{center}
\end{figure}

\begin{lemma}[\cite{Pandey}, Lemma 3.5] \label{c-cmn}
Let $m_1,m_2 \geq 3$ be two integers and let $n=m_1+m_2-1$. Then $W(C_n)>W(C_{m_1,m_2}^n).$
\end{lemma}

 The graphs $L_{n,g}$ and $C_{m_1,m_2}^n$ are important for us. So we recall the known expressions for the distance of a vertex in these graphs and for their Wiener indices. Wiener index of $L_{n,g}$ can be found in \cite{Yu} (see Theorem 1.1) as following. 
\begin{equation}\label{WI_Lnk}
W(L_{n,g})=\begin{cases} 
\frac{g^3}{8}+(n-g)(\frac{n^2+ng+3g-1}{6}-\frac{g^2}{12})  &\mbox{ if g is even}, \\ 
\frac{g(g^2-1)}{8}+(n-g)(\frac{n^2+ng+3g-1}{6}-\frac{g^2}{12}-\frac{1}{4})  &\mbox{ if g is odd}.
\end{cases}
\end{equation}
Also the distance of the pendant vertex $v$ in $L_{n,g}$ can be computed using Lemma \ref {count} ($i$) as
\begin{equation}\label{D_Lnk-p}
D_{L_{n,g}}(v)=
\begin{cases}
\frac{g^2}{4}+\frac{(n-g)(n+g-1)}{2}~~~&\mbox{ if $g$ is even},\\
\frac{g^2-1}{4}+ \frac{(n-g)(n+g-1)}{2}~~~&\mbox{ if $g$ is odd}.
\end{cases}
\end{equation}

The Wiener index of $C_{m_1,m_2}^n$ when $n=m_1+m_2-1$ is computed in \cite{Pandey} (see proof of Lemma 3.5 in \cite{Pandey}). We generalise that by counting the Wiener index of $C_{m_1, m_2}^n\in \mathfrak{C}_{n,k}$ i.e. when $n=m_1+m_2+k-2$. Let $w$ be the cut vertex of $C_{m_1, m_2}^n$ lying in $C_{m_1}$. Using Lemma \ref{count} ($ii$)
$$W(C_{m_1,m_2}^n)=W(C_{m_1})+W(L_{m_2+k-1, m_2})+(m_2+k-2)D_{C_{m_1}}(w)+(m_1-1)D_{L_{m_2+k-1,m_2}}(w).$$
From (\ref{WI_Lnk}), we get
{\small
\begin{align*}
W(L_{m_2+k-1,m_2})&=\begin{cases}
				       \frac{1}{24}(3m_2^3+6m_2^2k+4k^3+12m_2k^2-12k^2-6m^2-12m_2k+8k)\;&\mbox{ $m_2$ even},\\
				        \frac{1}{24}(3m_2^3+6m_2^2k+4k^3+12m_2k^2-12k^2-6m^2-12m_2k-3m_2+2k+6) &\mbox{ $m_2$ odd}.
                                        \end{cases}
\end{align*}}
Further, from (\ref{D_Lnk-p}), we get

\begin{align*}
D_{L_{m_2+k-1,m_2}}(w)=\begin{cases}
				\frac{1}{4}(m_2^2+4m_2k+2k^2-6k-4m_2+4) \;&\mbox{ $m_2$ even},\\
				\frac{1}{4}(m_2^2+4m_2k+2k^2-6k-4m_2+3) \;&\mbox{ $m_2$ odd}.\\
				 \end{cases}
\end{align*}
Performing some intricate calculations, the Wiener index of $C_{m_1,m_2}^n$ can be obtained as follows. \\

$W(C_{m_1, m_2}^n)$
{\small
\begin{align}\label{n=m_1+m_2+k-2}
=&\begin{cases}
	                    \frac{1}{8}(m_1^3+m_2^3+2m_1^2m_2+2m_1m_2^2+2m_1^2k+2m_2^2k+4m_1k^2+4m_2k^2+8m_1m_2k\\
\vspace{.1 cm}			-4m_1^2-4m_2^2- 8m_1m_2- 12m_1k- 12m_2k- 8k^2 +8m_1+ 8m_2-8)+\frac{1}{6}(k^3+11k) &\mbox{both $m_1, m_2$ even},\\
	                   \frac{1}{8}(m_1^3+m_2^3+2m_1^2m_2+2m_1m_2^2+2m_1^2k+2m_2^2k+4m_1k^2+4m_2k^2+8m_1m_2k\\
\vspace{.1 cm}			-4m_1^2-4m_2^2- 8m_1m_2- 12m_1k- 12m_2k- 8k^2 +6m_1+ 7m_2-4)+\frac{1}{12}(2k^3+19k) &\mbox{$m_1$ even , $m_2$ odd},\\
                          \frac{1}{8}(m_1^3+m_2^3+2m_1^2m_2+2m_1m_2^2+2m_1^2k+2m_2^2k+4m_1k^2+4m_2k^2+8m_1m_2k\\
\vspace{.1 cm}			-4m_1^2-4m_2^2- 8m_1m_2- 12m_1k- 12m_2k- 8k^2 +7m_1+ 6m_2-4)+\frac{1}{12}(2k^3+19k) &\mbox{$m_1$ odd, $m_2$ even},\\
	                  \frac{1}{8}(m_1^3+m_2^3+2m_1^2m_2+2m_1m_2^2+2m_1^2k+2m_2^2k+4m_1k^2+4m_2k^2+8m_1m_2k\\
			                -4m_1^2-4m_2^2- 8m_1m_2- 12m_1k- 12m_2k- 8k^2 +5m_1+ 5m_2)+\frac{1}{6}(k^3+8k) &\mbox{both $m_1, m_2$ odd}.
\end{cases}
\end{align}}

\section{s-pendant vertices and s-pendant blocks} \label{Special blocks}
We introduce a special kind of pendant blocks and pendant vertices in $\mathfrak{C}_{n,k}$, $k\geq 2$, which are used later in Section \ref{Wiener index}. 

\begin{definition}
Let $G$ be a graph with at least two cut vertices.  An {\bf s-pendant block} of $G$ is a pendant block which shares its cut vertex with exactly one non-pendant block of $G$. If size of an s-pendant block is $2$, then it is called an {\bf s-pendant edge}.
A pendant vertex lying on an s-pendant edge is called an  {\bf s-pendant vertex}. 
\end{definition}

In Figure \ref{s-pendant}, $B_2$ and $e_4$ are non pendant blocks. $B_1$ is an s-pendant block, $v_3$ is an s-pendant vertex and $e_3$ is an s-pendant edge. $B_3$ is a pendant block but not s-pendant, $v_1, v_2$ are pendant vertices but not s-pendant and $e_1, e_2$ are pendant edges but not s-pendant.
\begin{figure}[h!]
\begin{center}
\begin{tikzpicture}
\filldraw(0,0) circle [radius=.5mm]--(1,1) circle [radius=.5mm]--(2,0) circle [radius=.5mm]--(3,1) circle [radius=.5mm]--(4,1) circle [radius=.5mm]--(5,0) circle [radius=.5mm]--(6,1) circle [radius=.5mm]--(7,0) circle [radius=.5mm];
\filldraw(1,-1) circle [radius=.5mm] (3,-1) circle [radius=.5mm] (4,-1) circle [radius=.5mm] (5,0) (6,-1) circle [radius=.5mm](7,0)  (4.7,1.2) circle [radius=.5mm](5.3,1.2) circle [radius=.5mm](5,-1) circle [radius=.5mm](5.5,-2) circle [radius=.5mm];
\draw(0,0)--(1,-1)--(2,0)--(3,-1)--(4,-1)--(5,0)--(6,-1)--(7,0) (4.7,1.2)--(5,0)--(5.3,1.2) (5,0)--(5,-1)--(5.5,-2);
\draw (1,0) node {$B_1$} (3.5,0) node {$B_2$} (6,0) node {$B_3$} (4.7,1. 4) node {$v_1$} (5.5,1.4) node {$v_2$} (5.5,-2.2) node {$v_3$} (4.65,.8) node {$e_1$} (5.4,.8) node {$e_2$} (5.4,-1.5) node {$e_3$} (5.2,-0.7) node{$e_4$};
\end{tikzpicture}
\caption{s-pendant blocks and s-pendant vertices in a graph} \label{s-pendant}
\end{center}
\end{figure}
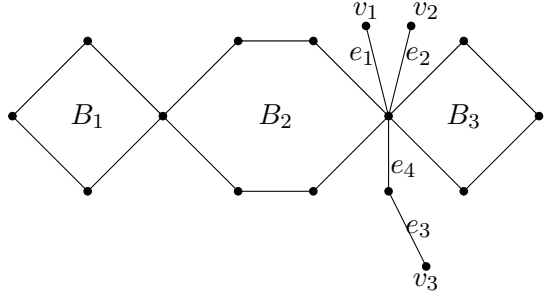

Let $G$ be a graph with $k\geq 2$ cut vertices. Let $w$ be a cut vertex and $B_1, B_2,\ldots, B_r$ be all the s-pendant blocks sharing the cut vertex $w$. Then the graph obtained from $G$ by detaching all the s-pendant blocks from $w$ i.e. $(G\setminus \cup_{i=1}^r{B_i})\cup w$  has $k-1$ cut vertices. The distance $d_G(B,B')$ or $d(B,B')$ between two blocks $B$ and $B'$ of $G$ is defined as $d_G(B,B')=\min\{d(u,v): u\in V(B), v\in V(B')\}$.
\begin{lemma}[\cite{Pandey1}, Lemma 2.5] \label {block_distance}
If $d(B_1,B_2)=\max\{d(B,B'): B, B' \mbox{ are blocks of}\; $G$\}$, then both $B_1$ and $B_2$ are pendant blocks in $G$.
\end{lemma}

\begin{proposition}\label{s-pendant blocks}
Let $G\in \mathfrak{C}_{n,k}$, $k\geq 2$. Then $G$ has at least two s-pendant blocks.
\end{proposition}
\begin{proof}
Let $B_1, B_2$ be two blocks in $G$ such that $d(B_1,B_2)=\max\{d(B,B'): B, B' \mbox{ are blocks of}\; $G$\}$.  By Lemma \ref {block_distance}, $B_1$ and $B_2$ are pendant blocks. We claim that both $B_1$ and $B_2$ are s-pendant blocks. Suppose $B_1$ is not s-pendant. Let $d(B_1, B_2)=d(w_1, w_2)$ where $w_1$ and $w_2$ are cut vertices lying in $B_1$ and $B_2$, respectively. Let $P_{w_1w_2}$ be a corresponding path joining $w_1$ and $w_2$. As $B_1$ is not s-pendant there exist a non-pendant block $B$ sharing the cut vertex $w_1$ with $B_1$ such that no vertex of $B$ other than $w_1$ lies on $P_{w_1w_2}$. Also as $B$ is non-pendant, $B$ is adjacent to a block $B'$ disjoint from $P_{w_1w_2}$. Let $B$ and $B'$ are adjacent via the cut vertex $w'$. Then $d(B_2, B')=d(w_2, w_1)+d(w_1,w')> d(B_1, B_2)$, a contradiction. Hence $B_1$ is s-pendant. Similarly $B_2$ is s-pendant.
\end{proof}

\section{Vertex peripherians attaining maximum distance in $\mathfrak{C}_{n,k}$}\label{Distance}
This section focuses on finding a graph that exhibit the vertex peripherians having maximum distance within the class $\mathfrak{C}_{n,k}$. For graphs devoid of cut vertices, the following is known.
 \begin{lemma}[\cite{Jp}, Lemma 2]\label {distance_v}
Let $G\in \mathfrak{C}_{n,0}$, $n\geq 3$ and $u\in V(G)$. Then for any $v\in V(C_n)$, $D_{C_n}(v) \geq D_{G}(u)$. Moreover
\begin{equation*}\label{eq-c2}
D_{C_n}(v)=\begin{cases}\frac{n^2}{4} &\textit{if n is even,} \\ 
\frac{n^2-1}{4} & \textit{if n is odd.}\end{cases}  
\end{equation*}
\end{lemma}
We generalise Lemma \ref{distance_v} for graphs with $k$ cut vertices i.e. find a $G_0\in \mathfrak{C}_{n,k}$ and $v_0\in V(G_0)$ such that $D_{G_0}(v_0)=\max\{D_G(v): G\in \mathfrak{C}_{n,k}, v\in V(G)\}$.
\begin{lemma}[\cite{Pandey}, Lemma 3.6]\label{Dmax_Lnk}
Let $u$ be the pendant vertex and $v$ be a non-pendant vertex of $L_{n,n-k}$, $k\geq 1$. Then $D_{L_{n,n-k}}(u)>D_{L_{n,n-k}}(v)$. 
\end{lemma}

\begin{lemma}\label {pendant block is cycle}
Let $G\in \mathfrak{C}_{n,k}$, $n\geq 5$ and $k\geq 1$.  Let $B$ be a pendant block in $G$ with $|V(B)|=m\geq 4$ containing the cut vertex $w$. Let $G'$ be the graph obtained from $G$ by replacing $B$ by $C_m$ such that $w$ remains a cut vertex of $G'$ in $C_m$. Then for any $v\in V(G)$, there exists a $v'\in V(G')$ such that $D_{G'}(v')\geq D_G(v)$. In particular, if $v\in V(G)\setminus V(B)\cup\{w\}$, then $D_{G'}(v)\geq D_G(v)$.
\end{lemma}

\begin{proof}
Let $H$ be the subgraph of $G$ induced by $V(G)\setminus V(B)\cup\{w\}$. Then $G\cong H\cup B$ with $V(H)\cap V(B)=w$ and $G'\cong H\cup C_m$ with $V(H)\cap V(C_m)=w$.\\
If $v\in V(H)$, then by Lemma \ref{count} ($i$) we have
 \begin{align*}
 D_G(v)&=D_H(v)+(m-1)d(v,w)+D_B(w)\\
             &\leq D_H(v)+(m-1)d(v,w)+ D_{C_m}(w) &\mbox{ [using Lemma \ref{distance_v}]}\\
             &=D_{G'}(v).
 \end{align*}
If $v\in V(B)$, take $v'$ to be the vertex in the cycle $C_m$ of $G'$ such that $d(v',w)=\lfloor\frac{m}{2}\rfloor$. Then
 \begin{align*}
 D_G(v)&=D_B(v)+(|V(H)|-1)d(v,w)+D_H(w)\\
             &\leq D_{C_m}(v')+(|V(H)|-1)\floor*{\frac{m}{2}}+ D_H(w)\\
             &=D_{C_m}(v')+(|V(H)|-1)d(v',w)+ D_H(w)\\
             &=D_{G'}(v').
 \end{align*}
This completes the proof.
 \end{proof}
 
 \begin{proposition}\label{pendant_blocks}
For $n\geq 4$ and $k\geq 1$ there exists a triangle free graph $G_0\in \mathfrak{C}_{n,k}$ with all its pendant blocks are either $K_2$ or a cycle on at least $4$ vertices such that $D_{G_0}(v_0)=\max\{D_G(v): G\in \mathfrak{C}_{n,k}, k\geq 1, v\in V(G)\}$.
 \end{proposition}

\begin{proof}
Let  $D_{G'}(v')=\max\{D_G(v): G\in \mathfrak{C}_{n,k}, k\geq 1, v\in V(G)\}$. By Remark \ref{Remark_Distance}, we may assume that $G'$ is triangle free. Replace each pendant blocks $B$ of $G'$ which is not cyclic by a cycle of size $|B|$. Let the resulting graph be $G_0$. By Lemma \ref {pendant block is cycle}, there exists a $v_0\in V(G_0)$ such that $D_{G_0}(v_0)\geq D_{G'}(v')$. Since $G_0\in \mathfrak{C}_{n,k}$ and $D_{G'}(v')=\max\{D_G(v): G\in \mathfrak{C}_{n,k}, k\geq 1, v\in V(G)\}$, $D_{G_0}(v_0)= D_{G'}(v')$.
\end{proof}

\begin{lemma}\label{p-block to cycle}
Let $H$ be a graph containing a vertex $w$ and $m_1,m_2\geq 3$ be two integers such that $m_1+m_2-1=m$. Let $G_1$ be the graph obtained by identifying the cut vertex of $C_{m_1,m_2}^m$ with $w$, $G_2$ be the graph obtained by identifying the cut vertex of $L_{m,m-1}$ with $w$ and $G$ be the graph obtained by identifying a vertex of $C_m$ with $w$. Then for any $v\in V(H)$, $D_{G}(v)>D_{G_1}(v)$ and $D_{G}(v)> D_{G_2}(v)$.
\end{lemma}

\begin{proof}
\begin{align*}
D_{G_1}(v)&=D_H(v)+(m-1)d(v,w)+D_{C_{m_1, m_2}^m}(w)\\
           &\leq D_H(v)+(m-1)d(v,w)+\frac{m_1^2+m_2^2}{4} \hskip 1cm\mbox {[equality holds if both $m_1, m_2$ are even]}\\
           &< D_H(v)+(m-1)d(v,w)+\frac{(m_1+m_2-1)^2-1}{4}\\
           &\leq D_H(v)+(m-1)d(v,w)+D_{C_m}(w)\\
           &= D_G(v)
\end{align*}
and 
\begin{align*}
D_{G_2}(v)&=D_H(v)+(m-1)d(v,w)+D_{L_{m, m-1}^m}(w)\\
           &\leq D_H(v)+(m-1)d(v,w)+\frac{(m-1)^2}{4}+1 &\mbox { [equality holds if $m$ is odd]}\\
           &< D_H(v)+(m-1)d(v,w)+\frac{m^2-1}{4} &\mbox{ [for $m\geq 4$]}\\
           &\leq D_G(v).
\end{align*}
\end{proof}

\begin{lemma}\label {Comparision}
Let $v_0$ be the pendant vertex of $L_{n,n-k}$ and let $n=m_1+m_2+k-2$, $m_1,m_2\geq 3$ and $k\geq 1$. Then  $D_{L_{n, n-k}}(v_0)> D_{C_{m_1, m_2}^n}(v)$ for any $v\in V(C_{m_1, m_2}^n)$.
\end{lemma}

\begin{proof}
First suppose $k\geq 2$. Let $w$ and $w'$ be the cut vertices of $C_{m_1,m_2}^n$ lying in $C_{m_1}$ and $C_{m_2}$, respectively. Let $P_k: w=v_1 v_2\cdots v_k=w'$ be the path in $C_{m_1, m_2}^n$ joining $w$ and $w'$. Then $V(C_{m_1}), V(C_{m_2})$ and $V(P)\setminus\{w, w'\}$ partition $V(C_{m_1,m_2}^n)$. First suppose $v$ belongs to the part $V(C_{m_1})$. Then 
\begin{align*}
D_{C_{m_1, m_2}^n}(v)&=D_{C_{m_1}}(v)+(m_2+k-2)d(v,w)+D_{L_{m_2+k-1,m_2}}(w)\\
                                      &\leq \frac{m_1^2}{4}+(m_2+k-2)\frac{m_1}{2}+ \frac{m_2^2}{4}+\frac{(k-1)(2m_2+k-2)}{2}\\
                                      &=\frac{1}{4}(m_1^2+m_2^2+2m_1m_2+2m_1k+4m_2k+2k^2-4m_1-4m_2-6k+4).
\end{align*}
Similarly if $v$ belongs to the part $V(C_{m_2})$, we get
\begin{align*}
D_{C_{m_1, m_2}^n}(v)&\leq \frac{1}{4}(m_1^2+m_2^2+2m_1m_2+4m_1k+2m_2k+2k^2-4m_1-4m_2-6k+4).
\end{align*}
Finally suppose $v\in V(P)\setminus\{w,w'\}$. Let $v=v_i, 2\leq i\leq k-1$. W.L.O.G. assume that $m_1\leq m_2$. Then we have
\begin{align*}
D_{C_{m_1, m_2}^n}(v)&=D_{L_{m_2+k-1,m_2}}(v_i)+(i-1)(m_1-1)+D_{C_{m_1}}(w)\\
				    &=D_{P_k}(v_i)+(k-i)(m_2-1)+D_{C_{m_2}}(w')+(i-1)(m_1-1)+D_{C_{m_1}}(w)\\
				    &\leq \frac{k(k-1)}{2}+ (m_2-1)(k-1)+\frac{m_2^2}{4}+\frac{m_1^2}{4}  &[\mbox{using \ref{Path_vertex}}]\\
				    &=\frac{1}{4}(m_1^2+m_2^2+4m_2k+2k^2-4m_2-6k+4).
\end{align*} 
Further, we have
\begin{align*}
D_{L_{n, n-k}}(v_0)&\geq \frac{(n-k)^2-1}{4}+ \frac{k(2n-k-1)}{2} &\mbox{[from (\ref{D_Lnk-p})]}\\
                              &=\frac{1}{4}(m_1^2+m_2^2+2m_1m_2+4m_1k+4m_2k+2k^2-4m_1-4m_2-10k+3).
\end{align*}
By comparing, we get $D_{L_{n, n-k}}(v_0)>D_{C_{m_1, m_2}^n}(v)$ for any $v\in V(C_{m_1, m_2}^n)$.

Now suppose $k=1$. Then $w=w'$ and $V(P)\setminus\{w,w'\}=\emptyset$. By similar calculations as above it follows that $D_{L_{n, n-k}}(v_0)>D_{C_{m_1, m_2}^n}(v)$ for any $v\in V(C_{m_1, m_2}^n)$. 
\end{proof}

\begin{lemma}\label{Dmax@pendant}
If $D_{G_0}(v_0)=\max\{D_G(v): v\in V(G), G\in \mathfrak{C}_{n,k}, k\geq 1\}$, then $v_0$ must lie in some pendant block of $G_0$. Furthermore, if $n\geq 5$, then the pendant block containing $v_0$ shares its cut vertex with exactly one other block.  
\end{lemma}
\begin{proof}

Let $D_{G_0}(v_0)=\max\{D_G(v): v\in V(G), G\in \mathfrak{C}_{n,k}, k\geq 1\}$. We consider the following two cases.

\noindent{\bf Case I:} $k=1$\\
In this case every block of $G_0$ is a pendant block and hence $v_0$ lies in a pendant block.  Suppose $n\geq 5$ and $G_0$ has more than two blocks. Let $B_1, B_2, \ldots, B_s$ be the blocks where $s\geq 3$ and $|V(B_i)|=m_i$ for $1\leq i\leq s$. 
Assume that $v_0\in V(B_1)$ and $w$ is the cut vertex of $G_0$. If $m_i=2$ for $2\leq i\leq s$, then $n=m_1+s-1$ and 
\begin{align*}
D_{G_0}(v_0)&=D_{B_1}(v_0)+(s-1)d(v_0, w)+s-1\\
		     &\leq \frac{m_1^2}{4}+(s-1)(\frac{m_1}{2}+1)  & \mbox{[by Lemma \ref{distance_v}]}\\
		     &=\frac{1}{4}(m_1^2+2m_1s+4s-2m_1-4)\\
		     &\leq D_{C_{m_1,s}^n}(v_0) & \mbox{[equality holds for $s=3$]}\\
		     &< D_{L{n,n-1}}(v_0') & \mbox{[by Lemma \ref{Comparision}]}
\end{align*}
where $v_0'$ is the pendant vertex of $L_{n,n-1}$. This contradicts that $D_{G_0}(v_0)=\max\{D_G(v): v\in V(G), G\in \mathfrak{C}_{n,k}, k\geq 1\}$.

Other possibility is $m_i\geq 3$ for some $2\leq i\leq s$. For $2\leq i\leq s$, if $m_i\geq 3$ and $B_i\ncong C_{m_i}$, then replace $B_i$ by $C_{m_i}$ such that $w$ remains a cut vertex in the new graph. Let the resulting graph be  $G_0'$. Then by Lemma \ref{pendant block is cycle}, it follows that $D_{G_0'}(v_0)=D_{G_0}(v_0)$. W. L. O. G. assume that $m_2\geq 3$. Let $H'$ be the subgraph of $G_0'$ corresponding to $B_2\cup B_3$ in $G_0$. Then $H'\cong C_{m_2,m_3}^{m_2+m_3-1}$ or $H'\cong L_{m_2+m_3-1, m_2+m_3-2}$ and $G_0'\cong H\cup H'$ such that $V(H)\cap V(H')=\{w\}$. The subgraph $H$ contains $B_1$. Construct a new graph $G_0''$ from $G_0'$ by replacing $H'$ by $C_{m_2+m_3-1}$. Then by Lemma \ref{p-block to cycle}, $D_{G_0''}(v_0)> D_{G_0'}(v_0)=D_{G_0}(v_0)$, which is a contradiction. Hence $B_1$ shares its cut vertex with exactly one other block.\\
 
\noindent {\bf Case II:} $k\geq 2$\\
 Suppose $v_0$ lies in a non-pendant block $B$ of $G_0$. Then there are at least two cut vertices $w$ and $w'$ of $G_0$ in $B$. Assume $d(v_0,w)\geq d(v_0,w')$. $G_0$ has two subgraphs $G_1$ and $G_2$ such that $G_0\cong G_1\cup G_2$ with $V(G_1)\cap V(G_2)=\{w'\}$. Assume that $B$ lies in $G_1$. Let $z\in V(G_1)$ such that $d(v_0,z)=\max\{d(v_0,v): v\in V(G_1)\}$. Note that $z$ can not be a cut vertex of $G_0$ and $d(v_0, w)< d(v_0,z)$. Construct a new graph $G_0' \in \mathfrak{C}_{n, k}$ from $G_1$ and $G_2$ by identifying $z$ and $w'$ (of $G_2$). i.e. $G_0'\cong G_1\cup G_2$ and $V(G_1)\cap V(G_2)=\{z\}$. By Lemma \ref{count} ($i$), 
\begin{align*}
D_{G_0}(v_0)&=D_{G_1}(v_0)+(|V(G_2)|-1)d(v_0,w')+D_{G_2}(w')\\
		     &< D_{G_1}(v_0)+(|V(G_2)|-1)d(v_0,z)+D_{G_2}(w')\\
		     &=D_{G_0'}(v_0).
\end{align*}
which is a contradiction. So $B$ must be pendant. 

Let $w$ be the cut vertex of $G_0$ in $B$. Suppose the cut vertex $w$ is shared by at least two other blocks $B_1$ and $B_2$. Then there exist two subgraphs $G_1$ and $G_2$ of $G_0$ such that $G_0\cong G_1\cup G_2$ and $V(G_1)\cap V(G_2)=\{w\}$ where $G_1$ contains $B$ and has at least $2$ cut vertices. Let $w_1$ be  a cut vertex in $G_1\setminus B$ at maximum distance from $v_0$. Construct $G_0'\in \mathfrak{C}_{n,k}$ from $G_1$ and $G_2$ by identifying  $w_1$ and $w$ (of $G_2$) i.e. $G_0'\cong G_1\cup G_2$ and $V(G_1)\cap V(G_2)=\{w_1\}$. Then by Lemma \ref{count} ($i$), it follows that $D_{G_0'}(v_0)> D_{G_0}(v_0)$, which is a contradiction. Hence  $B$ shares its cut vertex with exactly one other block. This completes the proof.
\end{proof}

\begin{theorem} \label{Dmax}
Let $z$ be the pendant vertex of $L_{n,n-k}$ and $G\in \mathfrak{C}_{n,k}$,  $1\leq k\leq n-3$. Then $D_G(v)\leq D_{L_{n,n-k}}(z)$ for any $v\in V(G)$. 
\end{theorem}

\begin{proof}
We use induction on the number of cut vertices $k$.\\

\noindent {\bf Base step:} $\mathfrak{C}_{4,1}=\{L_{4,3}, K_{1,3}\}$ and the result can easily be verified for $n=4$. Now suppose  $n\geq 5$. By Proposition \ref{pendant_blocks}, there exists a triangle free graph $G_0\in \mathfrak{C}_{n,1}$ having its blocks either $K_2$ or cycles on at least $4$ vertices such that $D_{G_0}(v_0)=\max\{D_G(v):v\in V(G), G\in\mathfrak{C}_{n,1}\}$. By Lemma \ref{Dmax@pendant}, either $G_0\cong C_{m_1,m_2}^n$ for some $m_1, m_2$ satisfying $m_1+m_2-1=n$ or $G_0\cong L_{n,n-1}$. But by Lemma \ref{Comparision}, $D_{L_{n,n-1}}(z)> D_{C_{m_1,m_2}^n}(v)$ for any $v\in V(C_{m_1,m_2}^n)$ where $z$ is the pendant vertex of $L_{n,n-1}$. Hence $G_0\cong L_{n,n-1}$. Now by Lemma \ref {Dmax_Lnk}, $v_0$ is the pendant vertex of $G_0 (\cong L_{n,n-1})$. So the result is true for $k=1$.\\

\noindent{\bf Induction hypothesis:} Let the result be true for $G\in \mathfrak{C}_{n,k-1}$, $k\geq 2$ and $n\geq k+2$. \\

\noindent{\bf Induction step:} By Proposition \ref{pendant_blocks}, there exists a triangle free graph $G_0\in \mathfrak{C}_{n,k}$ having each of its pendant blocks is either $K_2$ or a cycle on at least $4$ vertices such that $D_{G_0}(v_0)=\max\{D_G(v):v\in V(G), G\in\mathfrak{C}_{n,k}\}$. It is sufficient to show that $D_{L_{n,n-k}}(z)=D_{G_0}(v_0)$. As $k\geq 2$, $n\geq 5$. By Lemma \ref{Dmax@pendant}, $v_0$ belongs to some pendant block $B$ of $G_0$ and $B$ shares its cut vertex with exactly one other block. Let $w$ be the cut vertex of $G_0$ in $B$. Then $G_0\cong B \cup H$ with $V(B)\cap V(H)=\{w\}$. Let $|V(B)|=n_1$ and $|V(H)|=n_2=n-n_1+1$. Then $H\in \mathfrak{C}_{n_2, k-1}$. So $n_2\geq k+1$. \\

Suppose $n_2=k+1$, then $H$ is a path and $B\cong C_{n_1}, n_1\geq 4$. So $G_0\cong L_{n,n_1}$ and we get $D_{L_{n,n_1}}(v_0)>D_{L_{n,n_1}}(z)$, where $v_0$ is a non pendant vertex of $L_{n,n_1}$. This is a contradiction to Lemma \ref{Dmax_Lnk}. So $n_2\geq k+2$. Suppose $w$ as a vertex of $H$ is not the pendant vertex of $L_{n_2, n_2-k+1}$ ($H$ may be isomorphic to $L_{n_2,n_2-k+1}$).  Construct a new graph $G_0' \in \mathfrak{C}_{n,k}$ from $G_0$ by replacing $H$ by $ L_{n_2, n_2-k+1}$ such that $G_0'\cong B\cup  L_{n_2, n_2-k+1}$ and $V(B)\cap V(L_{n_2, n_2-k+1})=\{w\}$ where $w$ is the pendant vertex of $L_{n_2, n_2-k+1}$. Then by induction hypothesis $D_{ L_{n_2, n_2-k+1}}(w)\geq D_H(w)$. By Lemma \ref {count} we have 
\begin{align*}
D_{G_0'}(v_0)&=D_B(v_0)+(n_2-1)d(v_0,w)+D_{L_{n_2, n_2-k+1}}(w)\\
		  &\geq D_B(v_0)+(n_2-1)d(v_0,w)+D_H(w)\\
		  &=D_{G_0}(v_0).
\end{align*}
As  $D_{G_0}(v_0)=\max\{D_G(v):v\in V(G), G\in\mathfrak{C}_{n,k}\}$, $D_{G_0'}(v_0)=D_{G_0}(v_0)$. \\

We now claim that $B\cong K_2$. Suppose $B\ncong K_2$. Then $B\cong C_{n_1}, n_1\geq 4$ and $G_0'\cong C_{n_1, n_3}^n$ where $n_3=n_2-k+1$. By Lemma \ref{Comparision}, $D_{G_0'}(v_0)< D_{L_{n, n-k}}(u_0)$ where $u_0$ is the pendant vertex of $L_{n,n-k}$, which is a contradiction. Therefore $B\cong K_2$ and hence $G_0' \cong L_{n,n-k}$.  Now by Lemma \ref{Dmax_Lnk}, $v_0$ is the pendant vertex $z$ of $L_{n,n-k}$. Thus $D_{L_{n,n-k}}(z)=D_{G_0'}(v_0)=D_{G_0}(v_0)$. This completes the induction step and hence the proof.
\end{proof}

\section {Maximum Wiener index over $\mathfrak{C}_{n,k}$}\label{Wiener index}
Let $G\in \mathfrak{C}_{n,0}$. Then $G$ is a $2$-connected graph and we have the following  characterisation of graphs having maximum Wiener index over $2$-connected graphs. 
 \begin{lemma}[\cite{Jp}, Theorem 5] \label{plesnic} 
Let  $G\in \mathfrak{C}_{n,0}$, $n\geq 3$. Then $W(G)\leq W(C_n)$ and equality holds if and only if $G\cong C_n.$ Furthermore,  
\begin{equation*}
 W(C_n)=\begin{cases}
 \frac{n^3}{8} &\mbox{ if $n$ is even,}\\
 \frac {n(n^2-1)}{8} &\mbox{ if $n$ is odd}.
 \end{cases}
 \end{equation*}
 \end{lemma}

\begin{lemma}\label{pendant blocks}
Let $G$ be a graph in $\mathfrak{C}_{n,k}$, $k\geq 1$ having maximum Wiener index. Then every pendant block of $G$ is either $K_2$ or a cycle on at least $4$ vertices.
\end{lemma}

\begin{proof}
Let $B$ be a pendant block in $G$ and let $w$ be the cut vertex of $G$ in $B$. Then there exists a subgraph $H$ (on at least $2$ vertices) of $G$ such that $G=H\cup B$ with $V(H)\cap V(B)=\{w\}$. Suppose $B\ncong K_2$. Then by Remark \ref{Remark_WI}, $|V(B)|\geq 4$. Let $|V(B)|=m$. Suppose $B\ncong C_m$. Construct a graph $G'$ from $G$ by replacing $B$ by $C_m$. Then using Lemma \ref{count} ($ii$), we get
\begin{align*}
W(G)&=W(H)+W(B)+(|V(H)|-1)D_B(w)+(m-1)D_H(w)\\
&\hskip -1cm\mbox{using Lemma \ref{distance_v} and Lemma \ref{plesnic}}\\
        &<W(H)+W(C_m)+(|V(H)|-1)D_{C_m}(w)+ (m-1)D_H(w)\\
        &=W(G'),
 \end{align*}
which is a contradiction. This completes the proof.
 \end{proof}
 
 \begin{lemma}\label{Cmn-Lnk}
Let $m_1,m_2 \geq 4$ be two integers such that $n=m_1+m_2-1$. Then $$W(C_n)\geq W(L_{n,n-1})>W(C_{m_1,m_2}^n).$$
\end{lemma}

\begin{proof}
By (\ref{n=m_1+m_2+k-2}), 

{\small
\begin{align*}\label{n=m_1+m_2-1}
W(C_{m_1,m_2}^n)&=
\begin{cases}
\frac{1}{8}(m_1^3+m_2^3+2m_2m_1^2+2m_1m_2^2-2m_1^2-2m_2^2) &\mbox{ both $m_1, m_2$ even},\\
\frac{1}{8}(m_1^3+m_2^3+2m_1^2m_2+2m_1m_2^2-2m_1^2-2m_2^2-2m_1-m_2+2) &\mbox{ $m_1$ even, $m_2$ odd},\\
\frac{1}{8}(m_1^3+m_2^3+2m_1^2m_2+2m_1m_2^2-2m_1^2-2m_2^2-m_1-2m_2+2) &\mbox{ $m_1$ odd, $m_2$ even}, \\
\frac{1}{8}(m_1^3+m_2^3+2m_2m_1^2+2m_1m_2^2-2m_1^2-2m_2^2-3m_1-3m_2+4)  &\mbox{ both $m_1, m_2$ odd}.
\end{cases}\\
& \leq \frac{1}{8}(m_1^3+m_2^3+2m_2m_1^2+2m_1m_2^2-2m_1^2-2m_2^2).
\end{align*}}
From (\ref{WI_Lnk}) we find
\begin{align*}
W(L_{n,n-1})&=\begin{cases}
			 \frac{1}{8}(n^3-n^2+6n-8) &\mbox{ if n is even,}\\	
 			  \frac{1}{8}(n^3-n^2+7n-7) &\mbox{ if n is odd.}	
                        \end{cases}\\
   		   &\geq \frac{1}{8}(n^3-n^2+6n-8)\\
		   &=\frac{1}{8}(m_1^3+m_2^3+3m_1^2m_2+3m_1m_2^2-4m_1^2-4m_2^2-8m_1m_2+11m_1+11m_2-16).
\end{align*}
So we get
\begin{align*}
W(L_{n,n-1})-W(C_{m_1,m_2}^n)&\geq m_1^2m_2+m_1m_2^2+11m_1+11m_2 - 2(m_1^2+m_2^2+4m_1m_2+8)\\
						   &=m_1m_2(m_1+m_2)+11(m_1+m_2)-2(m_1+m_2)^2-4m_1m_2-16\\
						   &=m_1m_2(m_1+m_2-4)+11(m_1+m_2)-2(m_1+m_2)^2-16\\
						   &\geq 2(m_1+m_2)(m_1+m_2-4)+11(m_1+m_2)-2(m_1+m_2)^2- 16\\ & \hskip -3cm \mbox{ (This inequality holds because $m_1+m_2-4>0$ and $m_1m_2\geq 2(m_1+m_2)$ for $m_1,m_2\geq 4$)}\\
						   &=3(m_1+m_2)-16\\
						   &>0   
\end{align*}
Further we have
\begin{align*}
W(C_n)-W(L_{n,n-1})&=\begin{cases}
					\frac{1}{8}(n^2-6n+8)  &\mbox{ if n is even,}\\
					\frac{1}{8}(n^2-8n+7)   &\mbox{ if n is odd}.	
                                     \end{cases}\\
                                     &\geq 0 \;\;\;\mbox{($>0$ for $n\geq 8$)}.
\end{align*}
This completes the proof.
\end{proof}

 \begin{proposition}\label {pendant-cut}
 Let $G$ has maximum Wiener index over $\mathfrak{C}_{n,k}, n \geq 7$. Then a cut vertex of $G$ is shared by at most two pendant blocks. Furthermore, for $k\geq 2$, if a cut vertex is shared by two pendant blocks then both of them are $K_2$.
 \end{proposition}
 
 \begin{proof}
 Let $B_1$ and $B_2$ be two pendant blocks of $G$ sharing the cut vertex $w$ such that $|V(B_1)|=m_1$ and $|V(B_2)|=m_2$. We consider the following two cases.\\
{\bf Case-I:} $n=m_1+m_2-1$\\
 By Lemma \ref {pendant blocks}, $G$ is isomorphic to either $L_{n,n-1}$ or to $C_{m_1,m_2}^n$. By Lemma \ref{Cmn-Lnk}, $W(L_{n,n-1})> W(C_{m_1,m_2}^n)$ and the result follows.\\
 
\noindent{\bf Case II:} $n> m_1+m_2-1$\\
In this cases there exists a subgraph $H$ (on at least two vertices) of $G$ such that  $G\cong H\cup (B_1\cup B_2)$ and $V(H)\cap V(B_1\cup B_2)=\{w\}$. First suppose both $m_1, m_2 >2$. Then by Lemma \ref{pendant blocks}, $m_1, m_2\geq 4$ and $B_1\cup B_2\cong C_{m_1,m_2}^{m_1+m_2-1}$. Construct a graph $G'$ from $G$ by replacing $B_1\cup B_2$ by $C_{m_1+m_2-1}$ such that $w$ remains a cut vertex shared by $H$ and $C_{m_1+m_2-1}$. Then
 \begin{align*}
 W(G)&=W(H)+W(C_{m_1,m_2}^{m_1+m_2-1})+(m_1+m_2-2)D_H(w)+(|V(H)|-1)D_{C_{m_1,m_2}^{m_1+m_2-1}}(w)\\
 & \mbox {using Lemma \ref{Cmn-Lnk}}\\
         &<W(H)+W(C_{m_1+m_2-1})+(m_1+m_2-2)D_H(w)+\frac{1}{4}(|V(H)|-1)(m_1^2+m_2^2)\\
         &<W(H)+W(C_{m_1+m_2-1})+(m_1+m_2-2)D_H(w)+\frac{1}{4}(|V(H)|-1)D_{C_{m_1+m_2-1}}(w)\\
         &=W(G'),
 \end{align*}
 which is a contradiction.
 
 Now suppose exactly one of $m_1$ or $m_2$ is greater than $2$. Let $m_1>2$ and $m_2=2$. Then by Lemma \ref{pendant blocks}, $B_1\cup B_2\cong L_{m_1+1,m_1}$ and $G\cong H\cup L_{m_1+1,m_1}$ such that $V(H)\cup V(L_{m_1+1,m_1})=\{w\}$, where $w$ is the cut vertex of $L_{m_1+1,m_1}$. Construct $G'$ from $G$ by replacing $B_1\cup B_2$ by $C_{m_1+1}$ such that $V(H)\cap V(C_{m_1+1})=\{w\}$. Then
\begin{align*}
W(G)&=W(H)+W(L_{m_1+1, m_1})+m_1D_H(w)+(|V(H)|-1)D_{L_{m_1+1,m_1}}(w)\\
 & \mbox {using Lemma \ref{Cmn-Lnk}}\\
	&\leq W(H)+W(C_{m_1+1})+m_1D_H(w)+\frac{1}{4}(|V(H)|-1)(m_1^2+4)\\
	&< W(H)+W(C_{m_1+1})+m_1D_H(w)+\frac{1}{4}(|V(H)|-1)((m_1+1)^2-1)\\
	&\leq W(H)+W(C_{m_1+1})+m_1D_H(w)+(|V(H)|-1)D_{C_{m_1+1}}(w)\\
	&=W(G'),
\end{align*}
 which is a contradiction. This shows that if $2$ or more pendant blocks share a cut vertex, then all of them are $K_2$. So, as $n\geq 7$, H can not be a block and hence $k\geq 2$. Suppose there are more than $2$ pendant blocks in $G$ isomorphic to $K_2$ sharing $w$. Then there exists a subgraph $H'$ (on at least $4$ vertices) of $G$ such that $G\cong H'\cup K_{1,3}$ such that $V(H')\cap V(K_{1,3})=\{w\}$. Construct $G'$ from $G$ by replacing $K_{1,3}$ by $C_4$ such that $V(H')\cap V(C_4)=\{w\}$. Then
 \begin{align*}
W(G)&=W(H')+W(K_{1,3})+3D_{H'}(w)+(|V(H')|-1)D_{K_{1,3}}(w)\\
	&= W(H')+9+3D_{H'}(w)+3(|V(H')|-1)\\
	&< W(H')+8+3D_{H'}(w)+4(|V(H')|-1) &\mbox{[as $|V(H')|\geq 4$]}\\
	&=W(G'),
\end{align*} 
which is a contradiction. This completes the proof.
 \end{proof}
 
For $4\leq n\leq 6$, the graphs having maximum Wiener index in $\mathfrak{C}_{n,1}$ have been shown in Table \ref{MaxWI_1}. The following theorem characterises the graphs having maximum Wiener index in $\mathfrak{C}_{n,1}$ for $n\geq 7$.
 
\begin{table}[h!]
\begin{center}
\begin{tabular}{|c|c|c|}
\hline
$n$&Graphs $G$ having maximum W.I. over $\mathfrak{C}_{n,1}$&$W(G)$\\
\hline
$4$&  \hskip 2cm\begin{tikzpicture}[scale=.5]
\filldraw (1,0) circle [radius=.5mm]--(2,0) circle [radius=.5mm]--(2.9,0.6) circle [radius=.5mm];
\filldraw (2,0)--(2.9,-0.6) circle [radius=.5mm];
\end{tikzpicture}
&$9$\\
\hline
$5$& $L_{5,4}$ and  \hskip 0.5cm \begin{tikzpicture}[scale=.5]
\filldraw (1.1,0.6) circle [radius=.5mm]--(2,0) circle [radius=.5mm]--(2.9,0.6) circle [radius=.5mm];
\filldraw (2,0)--(2.9,-0.6) circle [radius=.5mm];
\filldraw (1.1,-0.6) circle [radius=.5mm]--(2,0) ;
\end{tikzpicture}
&$16$\\
\hline
$6$& $L_{6,5}$ and  \hskip 0.5cm
 \begin{tikzpicture}[scale=.5]
\filldraw (1.1,0.6) circle [radius=.5mm]--(2,0) circle [radius=.5mm]--(2.9,0.6) circle [radius=.5mm]--(3.8,0) circle [radius=.5mm]--(2.9,-0.6) circle [radius=.5mm];
\draw (2,0)--(2.9,-0.6); 
\draw (1.1,-0.6)  circle [radius=.5mm]--(2,0);
\end{tikzpicture}
& $26$\\
\hline
\end{tabular}
\end{center}
\label{default}
\caption{ The graphs having maximum Wiener index over $\mathfrak{C}_{n,1}$, $4\leq n\leq 6$} \label{MaxWI_1}
\end{table}

\begin{theorem}\label{one cut vertex}
Let $G\in \mathfrak{C}_{n,1}, n\geq 7$. Then $W(G)\leq W(L_{n,n-1})$ and equality holds if and only if $G\cong L_{n,n-1}$.
\end{theorem}

\begin{proof}
Let $G'$ be a graph having maximum Wiener index in $\mathfrak{C}_{n,1}$. Then all the blocks of $G'$ are pendant. So by Lemma \ref{pendant blocks} and Proposition \ref{pendant-cut}, $G'\cong L_{n,n-1}$ or $G'\cong C_{m_1,m_2}^{m_1+m_2-1}$. By Lemma \ref{Cmn-Lnk}, $W(L_{n,n-1})>W(C_{m_1,m_2}^{m_1+m_2-1})$. So $G'\cong L_{n,n-1}$ and the result follows.
\end{proof} 

\begin{lemma}\label{Comparison}
Let $m_1, m_2\geq 4$ and $n=m_1+m_2$. Then $W(C_{m_1m_2}^n)\leq W(L_{n,n-2})$ and equality holds iff $m_1=m_2=4$.
\end{lemma}

\begin{proof}
Using $k=n-2$ in (\ref{WI_Lnk}) we get
\begin{equation*}
W(L_{n,n-2})=
\begin{cases}
\frac{1}{8}(n^3-2n^2+20n-32) \mbox{ if $n$ is even,}\\
\frac{1}{8}(n^3-2n^2+19n-34) \mbox{ if $n$ is odd.}
\end{cases}
\end{equation*}
Using $n=m_1+m_2$ and $k=2$ in (\ref{n=m_1+m_2+k-2}) the Wiener index of $C_{m_1, m_2}^n$ can be computed as 
\begin{equation*}
W(C_{m_1,m_2}^n)=
\begin{cases}
\frac {1}{8}(n^3-m_1m_2(n-8)) &\mbox{ if both $m_1$ and $m_2$ are even,}\\
\frac{1}{8}(n^3-m_1m_2(n-8)-n-m_1 ) &\mbox{ if $m_1$ is even and $m_2$ odd,}\\
\frac{1}{8}(n^3-m_1m_2(n-8)-n-m_2) &\mbox{ if $m_1$ is odd $m_2$ is even,}\\
\frac{1}{8}(n^3-m_1m_2(n-8)-3n)&\mbox{ if both $m_1$ and $m_2$ are odd.}
\end{cases}
\end{equation*}
First suppose $n$ is even. In this case 
$$W(C_{m_1,m_2}^n)\leq \frac {1}{8}(n^3-m_1m_2(n-8))$$ 
and equality holds if and only if both $m_1$ and $m_2$ are even. As $n\geq 8$, $m_1m_2\geq 2(m_1+m_2)=2n.$ Now
\begin{align*}
W(L_{n,n-2})-W(C_{m_1,m_2}^n)&\geq \frac{1}{8}(m_1m_2(n-8)-2n^2+20n-32) &\mbox{[equality iff both $m_1, m_2$ are even]}\\
						   &\geq \frac{1}{8}(2n(n-8)-2n^2+20n-32)\\
						   &\geq \frac{1}{8}(4n-32)\\
						   &\geq 0 
\end{align*}
and equality holds if and only if $m_1=m_2=4$. Now suppose $n$ is odd and hence $n\geq 9$. Without loss of generality assume that $m_1$ is even and $m_2$ is odd. Then 
\begin{align*}
W(L_{n,n-2})-W(C_{m_1,m_2}^n)&=\frac{1}{8}(m_1m_2(n-8)-2n^2+20n-34+m_1)\\
						   &\geq \frac{1}{8}(2n(n-8)-2n^2+20n-34+m_1)\\
						   &=\frac{1}{8}(4n-34+m_1)\\
						   &>0 
\end{align*}
This completes the proof.
\end{proof}

There are not many non-isomorphic graphs in $\mathfrak{C}_{5,2}$, $\mathfrak{C}_{6,2}$, $\mathfrak{C}_{7,2}$, $\mathfrak{C}_{8,2}$ and $\mathfrak{C}_{9,2}$. The graphs having maximum Wiener index over $\mathfrak{C}_{n,2}$, $5\leq n\leq 9$ have been shown in Table \ref{MaxWI_2}. The following theorem characterises the graph having maximum Wiener index in $\mathfrak{C}_{n,2}$ for $n\geq 10$.

 \begin{table}[h!]
\begin{center}
\begin{tabular}{|c|c|c|}
\hline
$n$&Graphs $G$ having maximum W.I. over $\mathfrak{C}_{n,2}$&$W(G)$\\
\hline
$5$& \hskip 2 cm\begin{tikzpicture}[scale=.5]
\filldraw (0,0) circle [radius=.5mm]--(1,0) circle [radius=.5mm]--(2,0) circle [radius=.5mm]--(2.9,0.6) circle [radius=.5mm];
\filldraw (2,0)--(2.9,-0.6) circle [radius=.5mm];
\end{tikzpicture}&
$18$\\
\hline
$6$& $L_{6,4}$ and \hskip .3cm \begin{tikzpicture}[scale=.5]
\filldraw (-0.9,0.6,0) circle [radius=.5mm]--(0,0) circle [radius=.5mm]--(1,0) circle [radius=.5mm]--(1.9,0.6) circle [radius=.5mm];
\filldraw (-0.9,-0.6) circle [radius=.5mm]--(0,0);
\filldraw (1,0)--(1.9,-0.6) circle [radius=.5mm];
\end{tikzpicture}
& $29$\\
\hline
$7$&\hskip 2cm \begin{tikzpicture}[scale=.5]
\filldraw (1,0) circle [radius=.5mm]--(2,0) circle [radius=.5mm]--(2.9,0.6) circle [radius=.5mm]--(3.8,0) circle [radius=.5mm]--(2.9,-0.6) circle [radius=.5mm];
\draw (2,0)--(2.9,-0.6); 
\filldraw (0.1,0.6)circle [radius=.5mm]--(1,0);
\filldraw(1,0)--(0.1,-0.6) circle [radius=.5mm];
\end{tikzpicture}
& $44$\\
\hline
$8$ &$C_{4,4}^8$ and $L_{8,6}$
& $64$\\
\hline
$9$& $L_{9,7}$ and \begin{tikzpicture}[scale=.5]
\filldraw (0.1,0.5) circle [radius=.5mm]--(1,0) circle [radius=.5mm]--(2,0) circle [radius=.5mm]--(2.9,0.6) circle [radius=.5mm]--(3.8,0.6) circle [radius=.5mm]--(4.7,0) circle [radius=.5mm]--(3.8,-0.6) circle [radius=.5mm]--(2.9,-0.6) circle [radius=.5mm];
\draw (2,0)--(2.9,-0.6); 
\draw (0.1,-0.5) circle [radius=.5mm]--(1,0);
\end{tikzpicture}
& $88$ \\
\hline
\end{tabular}
\end{center}
\label{default}
\caption{ The graphs having maximum Wiener index over $\mathfrak{C}_{n,2}$ for $5\leq n\leq 9$} \label{MaxWI_2}
\end{table}

\begin{theorem}\label{2 cut vertices}
Let $G\in \mathfrak{C}_{n,2}$, $n\geq 10$. Then $W(G)\leq W(L_{n,n-2})$ and the equality holds if and only if $G\cong L_{n,n-2}$. 
\end{theorem}
\begin{proof}
Let $G\in \mathfrak{C}_{n,2}$ and $w, w'$ be the two cut vertices in $G$. As we are maximizing the Wiener index, by Lemma \ref{pendant blocks} we can assume that each pendant block of $G$ is either $K_2$ or a cycle on at least $4$ vertices. Since $G\in \mathfrak{C}_{n,2}$, it has exactly one non-pendant block, say $B$. As $n\geq 10$, either the maximal connected subgraph containing $B$ and no other blocks at $w$ or the maximal connected subgraph containing $B$ and no other blocks at $w'$ has $6$ or more vertices. Let $H_1$ be the maximal connected subgraph containing $B$ and no other blocks at $w$ and assume that $|V(H_1)|=n_1\geq 6$. Take the union of pendant blocks at $w$ as $H_2$ and let $|V(H_2)|=n_2\geq 2$. Then $G\cong H_1\cup H_2$ and $V(H_1)\cap V(H_2)=\{w\}$. Also $H_1\in \mathfrak{C}_{n,1}$. Construct $G'$ from $G$ by replacing $H_1$ by $L_{n_1, n_1-1}$ such that $w$ is the pendant vertex and $w'$ is the cut vertex of $L_{n_2, n_2-1}$. Then we have
\begin{align*}
W(G)&=W(H_1)+W(H_2)+(n_2-1)D_{H_1}(w)+(n_1-1)D_{H_2}(w)\\
 	&\mbox {observing Theorem \ref{Dmax}, Theorem \ref{one cut vertex} and Table \ref{MaxWI_1}} \\
        &\leq W(L_{n_1,n_1-1})+W(H_2)+(n_2-1)D_{L_{n_1,n_1-1}}(w)+ (n_1-1)D_{H_2}(w)\\
        &=W(G')
\end{align*}
and equality holds iff $H_1\cong L_{n_1, n_1-1}$ and $w$ is the pendant vertex of $H_1$. $G'\cong H_2\cup L_{n_1,n_1-1}$ and $V(H_2)\cap V(L_{n_1,n_1-1})=\{w\}$. As we are maximizing the Wiener index, by Proposition \ref{pendant-cut}, we may assume that either $H_2\cong K_2$, $H_2\cong K_{1,2}$ with $w$ as its non pendant vertex or $H_2\cong C_{n_2}$ for $n_2\geq 4$. If $H_2\cong K_2$, then $G'\cong L_{n, n-2}$ and the result follows. For the remaining two possibilities of $H_2$, $G\ncong L_{n,n-2}$. \\

If $H_2 \cong K_{1,2}$ with $w$ as its non pendant vertex, then $n=n_1+2$ and $G' \cong K_{1,3}\cup C_{n_1-1}$ with $V(K_{1,3})\cap V(C_{n_1-1})=\{w'\}$ where $w'$ is a pendant vertex of $K_{1,3}$.  By counting we get 
$$W(G')\leq \frac{1}{8}(n_1^3+3n_1^2+31n_1-3)$$ and 
$$W(L_{n, n-2})=W(L_{n_1+2,n_1})\geq \frac{1}{8}(n_1^3+4n_1^2+23n_1+4).$$ 
This gives 
\begin{equation*}
W(L_{n,n-2})-W(G')\geq \frac{1}{8}(n_1^2-8n_1+7)>0. \hskip 1cm \mbox {[as $n_1\geq 8$]}
\end{equation*}
Thus $W(G')<W(L_{n, n-2})$ and hence  $W(G)<W(L_{n, n-2}).$\\

Finally, if $H_2\cong C_{n_2}$, $n_2\geq 4$, then $G'\cong C_{n_1-1,n_2}^n$. By Lemma \ref{Comparison}, $W(G')=W(C_{n_1-1, n_2}^n)< W(L_{n,n-2})$ and hence $W(G)<W(L_{n, n-2})$. This completes the proof.
\end{proof}

\begin{lemma} \label{Comparison_3}
Let $m_1,m_2\geq 4$ and $n=m_1+m_2+1$. If $n\geq 14$, then $W(C_{m_1,m_2}^n)< W(L_{n,n-3})$.
\end{lemma}

\begin{proof}
As $n=m_1+m_2+1$, $C_{m_1,m_2}^n\in \mathfrak{C}_{n,3}$. Using $k=3$ in (\ref{n=m_1+m_2+k-2}), we get  \\

$W(C_{m_1,m_2}^n)$
{\small
\begin{align*}
        &=\begin{cases}
	                    \frac{1}{8}(m_1^3+m_2^3+2m_1^2m_2+2m_1m_2^2+2m_1^2+2m_2^2+16m_1m_2+8m_1+8m_2 )         &\mbox{$m_1$ even , $m_2$ even,}\\
	                    \frac{1}{8}(m_1^3+m_2^3+2m_1^2m_2+2m_1m_2^2+2m_1^2+2m_2^2+16m_1m_2+6m_1+7m_2-2)	  &\mbox{$m_1$ even , $m_2$ odd,}\\
                            \frac{1}{8}(m_1^3+m_2^3+2m_1^2m_2+2m_1m_2^2+2m_1^2+2m_2^2+16m_1m_2+7m_1+6m_2-2)       &\mbox{$m_1$ odd, $m_2$ even,}\\
	               	   \frac{1}{8}(m_1^3+m_2^3+2m_1^2m_2+2m_1m_2^2+2m_1^2+2m_2^2+16m_1m_2+5m_1+5m_2-4 )	  &\mbox{$m_1$ odd, $m_2$ odd},
       \end{cases}
\end{align*}}
and for $n=m_1+m_2+1$, 
\begin{align*}
W(L_{n,n-3})&=\begin{cases}
               \frac{1}{8}(m_1^3+m_2^3+ 3m_1^2m_2+3m_1m_2^2+35m_1+35m_2-52)    \mbox{ if $n$ is even,}\\
               \frac{1}{8}(m_1^3+m_2^3+ 3m_1^2m_2+3m_1m_2^2+36m_1+36m_2-48)\mbox{ if $n$ is odd},
              \end{cases}
\end{align*}

Suppose $n$ is even. Note that in this case either $m_1$ is even and $m_2$ is odd, or $m_1$ is odd and $m_2$ is even. The resulting $C_{m_1,m_2}^n$ are isomorphic. So we may assume $m_1$ is even and $m_2$ is odd. This gives
\begin{align*}
W(L_{n,n-3})-W(C_{m_1, m_2}^n) &=\frac{1}{8}(m_1^2m_2+m_1m_2^2-2m_1^2-2m_2^2-16m_1m_2+29m_1+28m_2-50)\\
						    &= \frac{1}{8}(m_1m_2(m_1+m_2-12)-2(m_1+m_2)^2+28(m_1+m_2)+m_1-52)\\
						    &> \frac{1}{8}(2(m_1+m_2)(m_1+m_2-12)-2(m_1+m_2)^2+28(m_1+m_2)+m_1-52)\\
						    &\mbox{[Inequality holds because $m_1+m_2-12>0$ and here $m_1m_2>2(m_1+m_2)$]}\\
						    &= \frac{1}{8}(4(m_1+m_2)+m_1-50)\\
						    &>0. \hskip 6cm\mbox{[as $m_1+m_2\geq 13$]}
\end{align*}

Now suppose $n$ is odd. In this case $m_1, m_2$ are either both even or both odd. $W(C_{m_1,m_2}^n)$ is maximum when both $m_1, m_2$ are even. So
$$W(C_{m_1,m_2}^n) \leq \frac{1}{8}(m_1^3+m_2^3+2m_1^2m_2+2m_1m_2^2+2m_1^2+2m_2^2+16m_1m_2+8m_1+8m_2 ).$$
Similar counting as in the $n$ even case gives
\begin{equation*}
W(L_{n,n-3})-W(C_{m_1, m_2}^n) > \frac{1}{8}(4(m_1+m_2)-48)> 0. \hskip 1cm\mbox{[as $m_1+m_2 \geq 14$]}
\end{equation*}
This completes the proof.
\end{proof}

\begin{theorem}\label{3 cut vertices}
Let $ G\in \mathfrak{C}_{n,3}$, $n\geq 14$. Then $W(G)\leq W(L_{n,n-3})$ and equality holds if and only if $G\cong L_{n,n-3}$. 
\end{theorem}

\begin{proof}
Let $G\in \mathfrak{C}_{n,3}$. As $n\geq 14$, there exists an s-pendant block $B$ in $G$ such that $n-|V(B)|+1\geq 9$. By Lemma \ref{pendant blocks}, we may assume that $B$ is either $K_2$ or a cycle on at least $4$ vertices. Let $|V(B)|=n_1$ and $w$ be the cut vertex of $G$ in $B$. \\

If $B$ does not share its cut vertex with any other pendant block, then there exists $H\in \mathfrak{C}_{n-n_1+1,2}$ such that $G\cong B\cup H$ and $V(B)\cap V(H)=\{w\}$. Construct a graph $G'$ from $G$ by replacing $H$ by $L_{n-n_1+1,n-n_1-1}$ such that $w$ is the pendant vertex of $L_{n-n_1+1,n-n_1-1}$. Then 
\begin{align*}
W(G)&= W(B)+W(H)+(|V(H)|-1)D_{B}(w)+(|V(B)|-1)D_H(w)\\
         &\leq W(B)+W(L_{n-n_1+1,n-n_1-1})+ (|V(H)-1|)D_B(w)+(|V(B)|-1)D_{L_{n-n_1+1,n-n_1-1}}(w)\\
         &\mbox { [inequality follows from Theorem \ref{Dmax} and Theorem \ref{2 cut vertices}]}\\
         &= W(G')
         \end{align*}
and equality holds if and only if $H\cong L_{n-n_1+1, n-n_1-1}$ and $w$ is the pendant vertex of $H$. If $B\cong K_2$, then $G'\cong L_{n,n-3}$ and the result follows. If $B\cong C_{n_1}$, $n_1\geq 4$, then $G'\cong C_{n_1, n_2}^n$ where $n_1+n_2+1=n$. By Lemma \ref{Comparison_3}, $W(G')< W(L_{n,n-3})$ and hence $W(G)< W(L_{n,n-3})$.\\

If $B$ shares its cut vertex with another pendant block, then by Lemma \ref{pendant-cut}, we may assume that $B\cong K_2$ and it shares its cut vertex with another pendant block $K_2$. By similar argument as above we get $W(G)\leq W(G'')$ where $G''\cong K_{1,2}\cup L_{n-2, n-4}$ with $V(K_{1,2})\cap V(L_{n-2, n-4})=\{w\}$ where $w$ is the non pendant vertex of $K_{1,2}$. The Wiener indices of $G''$ and $L_{n,n-3}$ can be computed as
\begin{align*}
W(G'')&=\begin{cases}
               \frac{1}{8}(n^3-4n^2+56n-152)    \mbox{ if $n$ is even,}\\
               \frac{1}{8}(n^3-4n^2+55n-156)    \mbox{ if $n$ is odd,}
              \end{cases}
\end{align*}
and
\begin{align*}
W(L_{n,n-3})&=\begin{cases}
               \frac{1}{8}(n^3-3n^2+38n-88)    \mbox{ if $n$ is even,}\\
               \frac{1}{8}(n^3-3n^2+39n-85)    \mbox{ if $n$ is odd.}
              \end{cases}
\end{align*}
So
\begin{align*}
W(L_{n,n-3})-W(G'')&=\begin{cases}
         		         \frac{1}{8}(n^2-18n+64)    \mbox{ if $n$ is even,}\\
                                  \frac{1}{8}(n^2-16n+71)    \mbox{ if $n$ is odd,}
                                  \end{cases}\\
                               &>0. \hskip 2cm\mbox{[as $n\geq 14$]}
\end{align*}
Thus $W(G'')< W(L_{n,n-3})$ and hence $W(G)< W(L_{n,n-3})$. This completes the proof.
\end{proof}
 
In Theorem \ref{3 cut vertices}, the graphs maximizing the Wiener index over $\mathfrak{C}_{n,3}$ for $n\geq 14$ have been characterised. The graphs having maximum Wiener index over $\mathfrak{C}_{n,3}$, $6\leq n\leq 13$ have been listed in Table \ref{MaxWI_3}.

 
 \begin{table}[h!]
\begin{center}
\begin{tabular}{|c|c|c|c|}
\hline
$n$&Graphs $G$ having maximum W.I. over $\mathfrak{C}_{n,3}$&$W(G)$\\
\hline
$6$& 
\hskip 2cm \begin{tikzpicture}[scale=.5]
\filldraw (-1,0) circle [radius=.5mm]--(0,0) circle [radius=.5mm]--(1,0) circle [radius=.5mm]--(2,0) circle [radius=.5mm]--(2.9,0.6) circle [radius=.5mm];
\filldraw (2,0)--(2.9,-0.6) circle [radius=.5mm];
\end{tikzpicture}
&$32$\\
\hline
$7$& $L_{7,4}$ and 
\hskip .3cm \begin{tikzpicture}[scale=.5]
\filldraw (0,0) circle [radius=.5mm]--(1,0) circle [radius=.5mm]--(2,0) circle [radius=.5mm]--(2.9,0.6) circle [radius=.5mm];
\filldraw (2,0)--(2.9,-0.6) circle [radius=.5mm];
\filldraw (-0.9,0.6) circle [radius=.5mm] (-0.9,-0.6) circle [radius=.5mm];
\draw (-0.9,0.6)--(0,0)--(-0.9,-0.6);
\end{tikzpicture}
& $48$\\
\hline 
$8$&
\hskip 2cm \begin{tikzpicture}[scale=.5]
\filldraw (0,0) circle [radius=.5mm]--(1,0) circle [radius=.5mm]--(2,0) circle [radius=.5mm]--(2.9,0.6) circle [radius=.5mm]--(3.8,0) circle [radius=.5mm]--(2.9,-0.6) circle [radius=.5mm];
\draw (2,0)--(2.9,-0.6); 
\filldraw (-0.9,0.6)circle [radius=.5mm]--(0,0);
\filldraw(0,0)--(-0.9,-0.6) circle [radius=.5mm];
\end{tikzpicture}
& $69$\\
\hline
$9$&$C_{4,4}^9$& $96$\\ 
\hline
$10$&$C_{4,5}^{10}$ and 
\hskip .3cm \begin{tikzpicture}[scale=.5]
\filldraw (0,0) circle [radius=.5mm]--(1,0) circle [radius=.5mm]--(2,0) circle [radius=.5mm]--(2.9,0.6) circle [radius=.5mm]--(3.8,0.6) circle [radius=.5mm]--(4.7,0) circle [radius=.5mm]--(3.8,-0.6) circle [radius=.5mm]--(2.9,-0.6) circle [radius=.5mm];
\draw (2,0)--(2.9,-0.6); 
\filldraw (-0.9,0.6)circle [radius=.5mm]--(0,0);
\filldraw(0,0)--(-0.9,-0.6) circle [radius=.5mm];
\end{tikzpicture}
&$126$\\
\hline
$11$&$C_{4,6}^{11}$&$166$\\
\hline
$12$&$C_{4,7}^{12}$ and 
\hskip .3cm \begin{tikzpicture}[scale=.5]
\filldraw (0,0) circle [radius=.5mm]--(1,0) circle [radius=.5mm]--(2,0) circle [radius=.5mm]--(2.9,0.6) circle [radius=.5mm]--(3.8,0.7) circle [radius=.5mm]--(4.7,0.6) circle [radius=.5mm]--(5.6,0) circle [radius=.5mm]--(4.7,-0.6) circle [radius=.5mm]--(3.8,-0.7) circle [radius=.5mm]--(2.9,-0.6) circle [radius=.5mm];
\draw (2,0)--(2.9,-0.6); 
\filldraw (-0.9,0.6)circle [radius=.5mm]--(0,0);
\filldraw(0,0)--(-0.9,-0.6) circle [radius=.5mm];
\end{tikzpicture}
& $209$\\
\hline
$13$ & $C_{6,6}^{13}$ and $L_{13,10}$&$264$\\
\hline
\end{tabular}
\end{center}
\label{default}
\caption{ The graphs having maximum Wiener index over $\mathfrak{C}_{n,3}$ for $6\leq n\leq 13$} \label{MaxWI_3}
\end{table}

In conclusion, the investigation into maximal graphs in $\mathfrak{C}_{n,k}$ for $k\geq 4$ presents intriguing avenues for further exploration. It is believed that there exists an integer $n_0^k$ such that the graph $L_{n,n-k}$ maximizes the Wiener index over $C_{n,k}$ for $n\geq n_0^k$. The comparison between $W(L_{n,n-k})$ and $W(C_{m_1,m_2}^n)$ for arbitrary $k$ may involve some intricate calculations. Additionally, exploring the determination of $m_1, m_2, k$ for which $W(C_{m_1,m_2}^n)$ is maximum (or minimum) over $\mathfrak{C}_{n,k}$ opens a new avenue for study, which could be helpful in complementing the current study.\\

\end{document}